\documentclass[12pt, a4paper]{amsart}

\usepackage{amsfonts,amssymb,amsmath,amsthm}
\usepackage{enumerate}

\usepackage{graphicx}
\usepackage{mathtools}
\usepackage{amssymb}
\usepackage{amsmath, amsthm}
\usepackage{amscd}
\usepackage{enumitem}
\usepackage[dvipsnames]{xcolor}
\usepackage[all]{xy}
\usepackage[margin=2.7cm]{geometry}
\usepackage{hyperref}
\hypersetup{hidelinks, hypertexnames=false}
\hypersetup{
    hypertexnames=false,
    colorlinks,
    linkcolor={red!30!black},
    citecolor={blue!50!black},
    urlcolor={blue!80!black}
}

\newtheorem{theorem}[subsection]{Theorem}
\newtheorem{lemma}[subsection]{Lemma}
\newtheorem{sublemma}[subsubsection]{Lemma}

\newtheorem{corollary}[subsection]{Corollary}

\newtheorem{definition}[subsection]{Definition}

\newtheorem{remark}[subsection]{Remark}

\makeatletter
\@addtoreset{subsection}{section}
\@addtoreset{equation}{section}
\@addtoreset{figure}{section}
\@addtoreset{table}{section}
\makeatother

\makeatletter
\newcommand\testshape{family=\f@family; series=\f@series; shape=\f@shape.}
\def\myemphInternal#1{\if n\f@shape%
\begingroup\itshape #1\endgroup\/%
\else\begingroup\it\sffamily #1\endgroup%
\fi}
\def\myemph{\futurelet\testchar\MaybeOptArgmyemph}
\def\MaybeOptArgmyemph{\ifx[\testchar \let\next\OptArgmyemph
                 \else \let\next\NoOptArgmyemph \fi \next}
\def\OptArgmyemph[#1]#2{\index{#1}\myemphInternal{#2}}
\def\NoOptArgmyemph#1{\myemphInternal{#1}}
\makeatother

\newcommand\id{\mathrm{id}}          

\newcommand\eps{\varepsilon}

\newcommand\bR{\mathbb{R}}

\newcommand\bN{\mathbb{N}}

\newcommand\Diff{\mathcal{D}}       
\newcommand\Stab{\mathcal{S}}       

\newcommand\DiffId{\Diff_{\id}}     
\newcommand\StabId{\Stab_{\id}}     

\newcommand\Cinfty{\mathcal{C}^{\infty}}
\newcommand\Ci[2]{\mathcal{C}^{\infty}(#1,#2)}               

\newcommand\Stabilizer[1]{\Stab(#1)}
\newcommand\StabilizerId[1]{\StabId(#1)}

\newcommand\SO{\mathrm{SO}}

\newcommand\Aman{A}
\newcommand\Bman{B}
\newcommand\Cman{C}

\newcommand\Eman{E}

\newcommand\Gman{G}

\newcommand\Kman{K}

\newcommand\Mman{M}
\newcommand\Nman{N}

\newcommand\Pman{P}

\newcommand\Tman{T}
\newcommand\Uman{U}
\newcommand\Vman{V}
\newcommand\Wman{W}
\newcommand\Xman{X}
\newcommand\Yman{Y}
\newcommand\Zman{Z}

\newcommand\cov[1]{\widetilde{#1}} %

\newcommand\VV{\mathcal{V}}
\newcommand\WW{\mathcal{W}}

\newcommand\DiffM{\Diff(\Mman)}

\newcommand\DiffMX{\Diff(\Mman,\Xman)}

\newcommand\func{f}
\newcommand\dif{h}
\newcommand\gdif{g}

\newcommand\fSing{\Sigma_{\func}}

\newcommand\msect{\xi}

\newcommand\rsect{\sigma}
\newcommand\Cf[2]{\mathcal{E}(#1; #2)}
\newcommand\Df[2]{\Diff(#1; #2)}
\newcommand\CfU[3]{\mathcal{E}(#2,#2\setminus #1; #3)}
\newcommand\DfU[3]{\Diff(#2,#2\setminus #1; #3)}

\newcommand\conn{\nabla}
\newcommand\VE{V\Eman}
\newcommand\lift[2]{\widetilde{#1}_{#2}}

\newcommand\aHmt{F}
\newcommand\bHmt{\mathbf{\aHmt}}

\newcommand\aSet{\Yman}
\newcommand\bSet{\Zman}

\newcommand\VBundle[1]{V#1}
\newcommand\HBundle[1]{\ker(#1)}

\newcommand\VU{\VBundle{\aSet}}
\newcommand\HU{\HBundle{\conn}}

\newcommand\ball[2]{\Bman_{#2}(#1)}

\newcommand\hUman{\cov{\Uman}}
\newcommand\hVman{\cov{\Vman}}

\newcommand\Fol{\mathcal{F}}
\newcommand\Eclass{\mathcal{Z}}

\newcommand\CCman{\Cman}

\begin{document}


\title{Diffeomorphisms preserving Morse-Bott functions}
\author{Oleksandra Khokhliuk}
\email{khokhliyk@gmail.com}
\address{Department of Geometry, Topology and Dynamical Systems, \\
Taras Shevchenko National University of Kyiv, \\
Hlushkova Avenue, 4e, Kyiv, 03127 Ukraine}

\author{Sergiy Maksymenko}
\email{maks@imath.kiev.ua}
\address{Topology Laboratory, Institute of Mathematics of NAS of Ukraine, \\ Tereshchenkivska str., 3, Kyiv, 01004 Ukraine}

\keywords{diffeomorphism, singular foliation, Ehresmann connection, Morse-Bott function}

\subjclass[2010]{%
57R30, 
57R45, 
53B15, 
}

\begin{abstract}
Let $f:M\to\mathbb{R}$ be a Morse-Bott function on a closed manifold $M$, so the set $\Sigma_f$ of its critical points is a closed submanifold whose connected components may have distinct dimensions.
Denote by $\mathcal{S}(f) = \{h \in \mathcal{D}(M) \mid f\circ h=h \}$ the group of diffeomorphisms of $M$ preserving $f$ and let $\mathcal{D}(\Sigma_f)$ be the group of diffeomorphisms of $\Sigma_f$.
We prove that the ``restriction to $\Sigma_f$'' map $\rho:\mathcal{S}(f) \to \mathcal{D}(\Sigma_f)$, $\rho(h) = h|_{\Sigma_f}$, is a locally trivial fibration over its image $\rho(\mathcal{S}(f))$.
\end{abstract}

\maketitle

\section{Introduction}
Homotopy properties of groups of leaf-preserving diffeomorphisms for nonsingular foliations are studied in many papers, see e.g.
\cite{Rybicki:MonM:1995},
\cite{Rybicki:SJM:1996},
\cite{Rybicki:DM:1996},
\cite{Rybicki:DGA:1999},
\cite{Rybicki:DGA:2001},
\cite{HallerTeichmann:AGAG:2003},
\cite{AbeFukui:CEJM:2005},
\cite{LechRybicki:BCP:2007} and references therein.
Most of the results concern the extension of results by M.~Herman~\cite{Herman:CR:1971}, W.~Thurston~\cite{Thurston:BAMS:1974}, J.~Mather~\cite{Mather:Top:1971}, \cite{Mather:BAMS:1974},  D.~B.~A.~Epstein~\cite{Epstein:CompMath:1970} on proving perfectness of such groups.

However for singular foliations their groups of diffeomorphisms are less studied, e.g.~\cite{Fukui:JMKU:1980}, \cite{Rybicki:DM:1998}, \cite{Maksymenko:OsakaJM:2011}, \cite{LechMichalik:PMD:2013}.

The present paper is devoted to certain deformational properties of groups of leaf-preserving diffeomorphisms of codimension one foliations with Morse-Bott singularities.
These foliations, in particular, foliations by level sets of Morse-Bott functions, play an important role in Hamiltonian dynamics and Poisson geometry, see e.g.~\cite{Fomenko:UMN:1989}, \cite{ScarduaSeade:JDG:2009}, \cite{MafraScarduraSeade:JS:2014}, \cite{MartinezAlfaroMezaSarmientoOliveira:JDE:2016}, \cite{MartinezAlfaroMezaSarmientoOliveira:TMNA:2018}.

Let $\Mman$ be a smooth compact manifold and $\Fol$ be a Morse-Bott foliation on $\Mman$ such that every connected components of $\partial\Mman$ is a leaf of $\Fol$, see Definition~\ref{def:MB_fol}.
Thus the \myemph{singular} set $\Sigma$ of $\Fol$ is a disjoint union of finitely many closed submanifolds of $\Mman$.
A diffeomorphism $\dif:\Mman\to\Mman$ will be called \myemph{leaf preserving} for $\Fol$ whenever $\dif(\omega)=\omega$ for each leaf $\omega\in \Fol$.
Let $\Diff(\Mman)$ and $\Diff(\Sigma)$ be the groups of all smooth diffeomorphisms of $\Mman$ and $\Sigma$ respectively, and $\Diff(\Fol)$ be the subgroup of $\Diff(\Mman)$ consisting of leaf-preserving diffeomorphisms of $\Fol$.

The well-known result by J.~Cerf~\cite{Cerf:BSMF:1961} and R.~Palais~\cite{Palais:CMH:1960} implies that the ``restriction to $\Sigma$'' map $\rho:\Diff(\Mman) \to \Diff(\Sigma)$, $\rho(\dif) = \dif|_{\Sigma}$, is a locally trivial fibration over its image $\rho(\Diff(\Mman))$.
Since $\rho$ is a homomorphism, the latter statement is equivalent to \myemph{existence of local sections of $\rho$}, i.e.\! if $\alpha = \dif|_{\Sigma} \in \Diff(\Sigma)$ for some $\dif\in\Diff(\Mman)$, then there exists an open neighborhood $\VV \subset\Diff(\Sigma)$ of $\alpha$ and a continuous map $\msect:\VV\to\Diff(\Mman)$ such that $\msect(\beta)|_{\Sigma} = \beta$ for all $\beta\in\VV$.

However, the authors did not find in the available literature that \myemph{$\rho|_{\Diff(\Fol)}:\Diff(\Fol) \to \Diff(\Sigma)$ is a locally trivial fibration as well}, i.e. that one can guarantee that $\msect(\beta)$ is also leaf preserving.

The aim of the present paper is to show that this is true for foliations that in a neighborhood of every singular leaf coincide with level sets of some Morse-Bott function, see Theorems~\ref{th:MB_func:ltf} and~\ref{th:MBfol:ltf}.
In particular, it will allow to relate homotopy groups of $\Diff(\Fol)$ and its subgroup $\Diff(\Fol,\Sigma)$, consisting of diffeomorphisms fixed on $\Sigma$, via homotopy groups of $\Diff(\Sigma)$, see Corollary~\ref{cor:long_seq_hom_groups}.
This would be useful when the dimension of $\Sigma$ is small.

\section{Preliminaries}\label{sect:prelim}
Throughout the paper the word ``smooth'' means $\Cinfty$, and the spaces of smooth maps are always endowed with the corresponding strong $\Cinfty$ topologies and their subspaces with the induces ones.

We will say that a not necessarily surjective map $\rho:\Eman\to\Bman$ is a \myemph{locally trivial fibration over its image} if the map $\rho:\Eman\to p(\Eman)$ is a locally trivial fibration.
It follows from the path lifting property for $\rho$ that $\rho(\Eman)$ is a union of certain path components of $\Bman$.
Moreover, if $\Eman_x$ is a path component of $\Eman$, then $\rho(\Eman_x)$ is a path component of $\Bman$ and the restriction map $p:\Eman_x\to p(\Eman_x)$ is a locally trivial fibration as well, see \ref{enum:lm:sect_hom:x:2} of Lemma~\ref{lm:sect_hom} below.

Let $\Mman$ be a smooth manifold.
By a \myemph{submanifold} of $\Mman$ will always mean an \myemph{embedded} submanifold.
For a subset $\Xman \subset\Mman$ denote by $\DiffMX$ the group of diffeomorphisms of $\Mman$ fixed on $\Xman$, and by $\DiffId(\Mman,\Xman)$ its identity path component.
If $\Fol$ is a (possibly singular) foliation on $\Mman$, then we will denote by $\Diff(\Fol,\Xman)$ the group of leaf-preserving diffeomorphisms of $\Mman$ fixed on $\Xman$, and by $\DiffId(\Fol,\Xman)$ its identity path component.
If $\Xman$ is empty, then we will omit it from notation. E.g. we write $\DiffM$ instead of $\Diff(\Mman,\varnothing)$ and $\Diff(\Fol)$ instead of $\Diff(\Fol,\varnothing)$.

A \myemph{restriction} $\Fol|_{\Xman}$ of a foliation $\Fol$ to a subset $\Xman \subset\Mman$ is a partition of $\Xman$ into connected components of $\Xman\cap\omega$ for all $\omega\in\Fol$.

Recall also that a family of mutually disjoint subsets $\{\Cman_i\}_{i\in\Lambda}$ of a topological space $\Mman$ is called \myemph{discrete} if each $\Cman_i$ has an open neighborhood $\Vman_i$ such that $\Vman_i\cap\Vman_j = \varnothing$ for $i\not=j\in\Lambda$.

\section{Morse-Bott maps}
Let $\Mman$ be a smooth manifold and $\Pman$ be either the real line $\bR$ or the circle $S^1$.
Notice also that there is a natural right action
\begin{align*}
&\nu:\Ci{\Mman}{\Pman} \times \DiffM \to \Ci{\Mman}{\Pman}, &
\nu(\func,\dif) &= \func\circ\dif.
\end{align*}
of the groups of diffemorphisms $\DiffM$ of $\Mman$ on the space $\Ci{\Mman}{\Pman}$ of smooth maps $\Mman\to\Pman$.
For $\func\in \Ci{\Mman}{\Pman}$ and a subset $\Xman\subset\Mman$ let
\begin{align*}
\Stabilizer{\func} &= \{ \dif\in\DiffM \mid \func\circ\dif = \func\},
&
\Stabilizer{\func,\Xman} &= \Stabilizer{\func} \cap \DiffMX
\end{align*}
be the \myemph{stabilizers} of $\func$ with respect to the above action of $\DiffM$ and the induced action of $\DiffMX$.
Let also $\StabilizerId{\func}$ and $\StabilizerId{\func,\Xman}$ be the identity path components of the corresponding stabilizers.

\begin{definition}\label{def:MB_map}
A smooth map $\func:\Mman\to\Pman$ will be called \myemph{Morse-Bott} if it satisfies the following conditions:
\begin{enumerate}[leftmargin=*, label={\rm(\arabic*)}, topsep=1ex, parsep=1ex]
\item\label{enum:MBF:constant_n_bd}
$\func$ takes a constant value at each connected component of $\partial\Mman$ and has no critical points in $\partial\Mman$;

\item\label{enum:MBF:Hess_non_deg}
the set $\fSing$ of critical points of $\func$ is a union of at most countable discrete family of boundaryless submanifolds $\{\Cman_i\}_{i\in\Lambda}$ and for each $x \in \Cman_i$ the Hessian of $\func$ at $x$ is non-degenerate in the direction transversal to $\Cman_i$.
\end{enumerate}
\end{definition}

If all critical points of a Morse-Bott map $\func$ are isolated, then $\func$ is called \myemph{Morse}.
The well known analogue of Morse lemma, e.g.~\cite[Theorem~2]{BanyagaHurtubise:EM:2004}, claims that condition~\ref{enum:MBF:Hess_non_deg} is equivalent to the following one:
{\em\begin{enumerate}[start=2, leftmargin=*, label={\rm(\arabic*$'$)}]
\item\label{enum:MBF:normal_form}
for every critical point $x\in\Mman$ of $\func$ there exist an open neighborhood $\Uman$ of $x$ in $\Mman$, an open neighborhood $\Vman$ of $\func(x)$ in $\Pman$, two open embeddings $\dif:\Uman\to\bR^{m}$ and $\phi:\Vman\to\bR$, and two integers $\lambda,k$ such that $0\le\lambda\le k\leq m$, $\dif(x)=0$, $\phi(\func(x))=0$, $\phi$ preserves orientation, and
\[
\phi\circ\func\circ\dif^{-1}(v_1,\ldots,v_m) = -v_1^2 - \cdots - v_{\lambda}^2 + v_{\lambda+1}^2 + \cdots + v_{k}^2.
\]
\end{enumerate}}

In~\cite{Maksymenko:OsakaJM:2011}, \cite{Maksymenko:AGAG:2006}, \cite{Maksymenko:UMZ:ENG:2012}, \cite{Maksymenko:TA:2018} the second author computed the homotopy types of $\StabilizerId{\func,\Xman}$ for a large class of smooth maps $\func:\Mman\to\Pman$ (which includes all Morse maps) on all compact surfaces $\Mman$, where $\Xman$ is a finite (possibly empty) collection of regular components of level sets of $\func$ and critical points.

\begin{theorem}\label{th:Stab_2dim}
{\rm \cite[Theorem~1.3]{Maksymenko:AGAG:2006}, \cite[Theorem~2.1]{Maksymenko:UMZ:ENG:2012}, \cite[Theorem~3.1]{Maksymenko:TA:2018}.}
Let $\Mman$ be a compact surface, $\func:\Mman\to\Pman$ be a Morse map, and $\Xman$ be a finite (possibly empty) collection of regular components of level sets of $\func$ and critical points.
Then $\StabilizerId{\func,\Xman}$ is contractible if and only if either of the following conditions holds true:
\begin{itemize}
\item $\Mman$ is non-orientable;
\item $\func$ has at least one saddle, i.e. a critical point of index $1$;
\item $\chi(\Mman) < |\Xman|$, e.g. this hold if $\Xman$ contains a regular component of some level set $\func^{-1}(c)$, $c\in\Pman$.
\end{itemize}
In all other cases $\StabilizerId{\func,\Xman}$ is homotopy equivalent to the circle.

Suppose $\Mman$ is orientable.
Let $\omega$ be any volume form on $\Mman$ and $\Stabilizer{\func,\omega}$ be the stabilizer of $\func$ with respect to the right action of the group of $\omega$-preserving diffeomorphisms.
Then its identity path component $\StabilizerId{\func,\omega}$ is an \myemph{abelian} group and the inclusion $\StabilizerId{\func,\omega} \subset \StabilizerId{\func}$ is a homotopy equivalence.
\qed
\end{theorem}

The technique of those papers is essentially two-dimensional and based on the observation that for orientable surfaces $\StabilizerId{\func,\Xman}$ is a subgroup of the group diffemorphisms preserving orbits of the Hamiltonian flow of $\func$.
Let us mention that the papers~\cite{Maksymenko:AGAG:2006}, \cite{Maksymenko:ProcIM:ENG:2010}, \cite{Maksymenko:DefFuncI:2014}, \cite{Kudryavtseva:MatSb:1999}, \cite{Kudryavtseva:SpecMF:VMU:2012}, \cite{Kudryavtseva:MathNotes:2012}, \cite{Kudryavtseva:MatSb:2013}, \cite{MaksymenkoFeshchenko:UMZ:ENG:2014}, \cite{MaksymenkoFeshchenko:MFAT:2015}, \cite{MaksymenkoFeshchenko:MS:2015} are devoted to computations of the homotopy types of orbits of Morse maps on compact surfaces with respect to the above actions.

\begin{theorem}\label{th:MB_func:ltf}
Let $\func:\Mman\to\Pman$ be a Morse-Bott map, and $\Cman_1,\ldots,\Cman_k$ be any finite collection of its \myemph{compact} critical submanifolds.
Let also $\Xman\subset\Mman \setminus \fSing$ be a closed (possibly empty) subset.
Then the maps
\begin{align}
\label{equ:restr_stab_MB}
&\rho: \Stabilizer{\func,\Xman} \to \Diff(\fSing), & \rho(\dif) &= \dif|_{\fSing}, \\
\label{equ:restr_stab0_MB}
&\rho_0: \StabilizerId{\func,\Xman} \to \DiffId(\fSing) \equiv \prod_{i=1}^{k} \DiffId(\Cman_i), & \rho_0(\dif)&= (\dif|_{\Cman_1}, \ldots, \dif|_{\Cman_k}),
\end{align}
are locally trivial fibrations over their images and the map $\rho_0$ is surjective.
\end{theorem}

This theorem can be regarded as a variant of the result by J.~Cerf~\cite{Cerf:BSMF:1961} and Palais~\cite{Palais:CMH:1960} on local triviality of restrictions to critical submanifolds of a Morse-Bott function $\func$ for $\func$-preserving diffeomorphisms.

In fact Theorem~\ref{th:MB_func:ltf} holds for more general classes of maps, see Theorem~\ref{th:func_ltf}.
On the other hand it says nothing for Morse maps, i.e.\! when each $\Cman_i$ is a point.
\begin{remark}\rm
Evidently,
\begin{align*}
\ker(\rho)&= \Stabilizer{\func,\Xman\cup\fSing}, &
\ker(\rho_0) &= \Diff(\Mman,\fSing) \cap \StabilizerId{\func,\Xman},
\end{align*}
and these groups are the fibres of the corresponding fibrations~\eqref{equ:restr_stab_MB} and~\eqref{equ:restr_stab0_MB}.
Notice that $\ker(\rho_0)$ differs from $\StabilizerId{\func,\Xman \cup \fSing}$ as one may expect at first glance.
In fact we have only the inclusion
\[\StabilizerId{\func,\Xman \cup \fSing} \ \subset \ \Diff(\Mman,\fSing) \cap \StabilizerId{\func,\Xman}.\]
\end{remark}

There are two standard applications of Theorem~\ref{th:MB_func:ltf}.
\begin{corollary}[Ambient isotopy extension for critical submanifolds]\label{cor:homotopy_lifting_prop}
Every smooth isotopy $\aHmt:\fSing\times[0,1] \to \fSing$ with $\aHmt_0=\id_{\fSing}$ extends to an isotopy $\aHmt:\Mman\times[0,1] \to \Mman$ such that $\aHmt_0=\id_{\Mman}$ and $\aHmt_t\in\Stabilizer{\func,\Xman}$ for all $t\in[0,1]$.
\end{corollary}

\begin{corollary}\label{cor:long_seq_hom_groups}
There is a long exact sequence of homotopy groups:
\begin{multline*}
\cdots
\to
\prod_{i=1}^{k} \pi_{q+1} \DiffId(\Cman_i)
\to \pi_q \StabilizerId{\func,\Xman\cup\fSing} \to \pi_q \StabilizerId{\func,\Xman} \to
\cdots  \\ \cdots \to
 \prod_{i=1}^{k} \pi_1 \DiffId(\Cman_i)
\to \pi_0 \Stabilizer{\func,\Xman\cup\fSing} \to \pi_0 \Stabilizer{\func,\Xman} \to
 \pi_0 \rho(\Stabilizer{\func,\Xman}) \to 1,
\end{multline*}
where the (omitted) base points are the corresponding identity maps.
\end{corollary}
\begin{proof}
This is a usual sequence of homotopy groups
\[
\cdots \to
\pi_{q} \bigl(\Stabilizer{\func,\Xman\cup\fSing}, \id_{\Mman}\bigr) \to
\pi_{q} \bigl(\Stabilizer{\func,\Xman}, \id_{\Mman}\bigr) \to
\pi_{q} \bigl(\Diff(\fSing), \id_{\fSing}\bigr) \to \cdots
\]
of the fibration
$\rho:\Stabilizer{\func,\Xman} \to  \rho\bigl( \Stabilizer{\func,\Xman} \bigr) \subset\Diff(\fSing)$.
We just replaced for $q\geq1$ each group with its identity path component and took to account that $\DiffId(\fSing)$ is contained in $\rho\bigl( \Stabilizer{\func,\Xman} \bigr)$, see \ref{enum:lm:sect_hom:x:2} of Lemma~\ref{lm:sect_hom} below, and is naturally isomorphic with $\prod\limits_{i=1}^{k} \DiffId(\Cman_i)$.
\end{proof}

\subsection{Homotopy types of diffeomorphisms groups}
The following table contains a description of homotopy types of $\DiffId$ for all connected closed manifolds of dimensions $1$ and $2$, see e.g. \cite{Smale:ProcAMS:1959}, \cite{Gramain:ASENS:1973}, \cite{EarleEells:DG:1970}.
We present this information just for the convenience of the reader.
This would allow to relate homotopy groups of $\StabilizerId{\func,\Xman}$ and $\StabilizerId{\func,\Xman\cup\fSing}$ via Corollary~\ref{cor:long_seq_hom_groups}.

\begin{center}
\begin{tabular}{|l|c|c|}\hline
$\Cman$ & $\dim$ & $\DiffId(\Cman)$ \\ \hline
$S^1$   & $1$    & $S^1$ \\\hline
$S^2$, $\bR P^2$ & $2$   & $\SO(3)$ \\
$S^1\times S^1$  &       & $S^1\times S^1$ \\
$D^2$, $S^1\times I$, Klein bottle  &       & $S^1$ \\
all other cases &       & point \\ \hline
\end{tabular}
\end{center}
In particular, this allows to get some information about the homotopy type of $\StabilizerId{\func}$ for Morse-Bott maps from $3$-manifolds as their critical submanifolds have dimensions $\leq 2$.

\medskip

The homotopy types of $\Diff(\Cman)$ for $3$-manifolds are studied not so well.
Recall that the \myemph{Smale conjecture} holds for a Riemannian $3$-manifold $\Cman$ whenever the natural inclusion $\mathrm{Isom}(\Cman) \subset  \Diff(\Cman)$ of its isometry group into the groups of diffeomorphisms is a homotopy equivalence.
This conjecture was verified e.g. for all closed hyperblic $3$-manifolds (D.~Gabai~\cite{Gabai:JDG:2001}),
the unit $3$-sphere $S^3$ with the standard metric (A.~Hatcher~\cite{Hatcher:AnnM:1983}), and more generally for many classes of elliptic $3$-manifolds, e.g.~\cite{HongKalliongisMcCulloughRubinstein:LMN:2012}.
Chapter~1 of the latter book \cite{HongKalliongisMcCulloughRubinstein:LMN:2012} also contains a comprehensive historical description and the current state of Smale conjecture.

Let us also mention two other results:
\begin{itemize}[leftmargin=3ex, itemsep=0.5ex]
\item $\Diff(S^1\times S^2)$ is homotopy equivalent to $O(2)\times O(3) \times \Omega O(3)$, where $\Omega O(3)$ is a loop space of $O(3)$, A.~Hatcher~\cite{Hatcher:ProcAMS:1981};
\item if $\Cman$ is a closed orientable Haken manifold that does not contain embedded projective plane $\bR P^2$, then $\DiffId(\Cman)$ is homotopy equivalent to one of the following spaces: a point, $S^1$, $S^1\times S^1\times S^1$, A.~Hatcher~\cite{Hatcher:Top:1976}.
\end{itemize}

In dimensions $\geq4$ the situation is very unclear and mostly concerned with computations of the mapping class group $\pi_0\Diff(\Cman)$ and finding nontrivial elements in homotopy groups $\pi_q\DiffId(\Cman)$, e.g. \cite{Novikov:IzvAN:1965}, \cite{Schultz:Top:1971}, \cite{Hajduk:BP:1978}, \cite{DwyerSzczarba:IJM:1983}.

\subsection{Structure of the paper}
In~\S\ref{sect:fol_MB_sing} we define foliations generated by Morse-Bott functions and formulate an analogue of Theorem~\ref{th:MB_func:ltf} for leaf-preserving diffeomorphisms for such foliations, see~Theorem~\ref{th:MBfol:ltf}.
In \S\ref{sect:ehresmann_conn} we recall a notion of an Ehresmann connection in a fibre bundle and some of their properties: convexity and homotopy lifting property, see Lemmas~\ref{lm:Ehresmann_conn_convex} and \ref{lm:local_homotopy_ext}.
\S\ref{sect:loc_triv_emb} is devoted to the proof of a well-known ``local triviality for restrictions to a submanifold'', Theorem~\ref{th:loc_triv_fibr}.
Our proof is a modification of E.~Lima~\cite{Lima:CMH:1964} and uses Ehresmann connections, see Lemma~\ref{lm:sect_at_id}.
In \S\ref{sect:fiber_wise_Morse} we show that in a tubular neighborhood of a critical submanifold of a Morse-Bott map $\func$ there exists an Ehresmann connection ``preserving $\func$'', in the sense that the differential $d\func$ annihilates horizontal bundle of that connection.
Finally in \S\ref{sect:diff_leaf_pres_fol} we prove Theorem~\ref{th:Zfol:ltf} extending Theorem~\ref{th:MBfol:ltf}, and in~\S\ref{sect:stabs_smooth_maps_MP} deduce from it Theorem~\ref{th:func_ltf} which extends Theorem~\ref{th:MB_func:ltf}.

\section{Morse-Bott foliations}\label{sect:fol_MB_sing}

\begin{definition}\label{def:MB_fol}{\rm cf.~\cite{ScarduaSeade:JDG:2009}.}
Let $\Fol$ be a partition of a smooth manifold $\Mman$ and $\{\Cman_i\}_{i\in\Lambda}$ be a \myemph{discrete} family of elements of $\Fol$.
We say that $\Fol$ is a \myemph{Morse-Bott foliation} with singularities $\{\Cman_i\}_{i\in\Lambda}$ whenever
\begin{enumerate}[label={\rm(\alph*)}, leftmargin=*]
\item\label{enum:MBfol:regular_out_of_sing}
the restriction of $\Fol$ to $\Mman\setminus \mathop{\cup}\limits_{i\in\Lambda}\Cman_i$
is a smooth codimension one foliation and every connected component of $\partial\Mman$ is an element of $\Fol$;

\item\label{enum:MBfol:Ci}
every $\Cman_i$ is a smooth boundaryless submanifold of $\Mman$ and $\dim\Cman_i<\dim\Mman$;

\item\label{enum:MBfol:local_normal_form}
for each $i\in\Lambda$ and $x\in\Cman_i$ there exist two numbers $\lambda \leq k \in \{0,\ldots,m\}$, an open neighborhood $\Uman$ of $x$ in $\Mman$, and an open embedding $\phi:\Uman\to\bR^{m}$ such that $\phi(x)=0$, and for every leaf $\omega$ of $\Fol$ with $\omega\cap\Uman\not=\varnothing$ and every connected component $\Kman$ of $\omega\cap\Uman$ the image $\phi(\Kman)$ is a connected component of some level set of the function $\func:\phi(\Uman)\to\bR$ defined by
\[
\func(v_1,\ldots,v_m) =
-v_1^2 - \cdots - v_{\lambda}^2 + v_{\lambda+1}^2 + \cdots + v_{k}^2
\]
for $(v_1,\ldots,v_m)\in\phi(\Uman)$.
\end{enumerate}
In that case each $\Cman_i$ will be called a \myemph{singluar} leaf of $\Fol$.
\end{definition}

The following lemma shows that every regularly embedded leaf of $\Fol$ can be regarded as a ``singular leaf of Morse-Bott type''.

\begin{lemma}
Let $\Fol$ be a Morse-Bott foliation on $\Mman$ with singularities $\{\Cman_i\}_{i\in\Lambda}$, $\Sigma = \mathop{\cup}\limits_{i\in\Lambda}\Cman_i$, and $\Cman$ be a leaf of $\Fol$ being also an embedded submanifold of $\Mman$ and having an open neighborhood $\Vman$ with $\overline{\Vman} \cap \Sigma = \varnothing$.
Then $\Fol$ is also a Morse-Bott foliation with singularities $\{\Cman\} \cup \{\Cman_i\}_{i\in\Lambda}$.
\end{lemma}
\begin{proof}
Evidently, $\{\Cman\} \cup \{\Cman_i\}_{i\in\Lambda}$ is a discrete family of submanifolds of $\Mman$ also satisfying conditions~\ref{enum:MBfol:regular_out_of_sing} and~\ref{enum:MBfol:Ci} of Definition~\ref{def:MB_fol}.
Therefore one need only to check condition~\ref{enum:MBfol:local_normal_form} for $\Cman$.
As $\Cman$ is a leaf of a codimension one foliation $\Fol|_{\Mman\setminus\Sigma}$, for each $x\in\Cman$ there exist an open neighborhood $\Uman$ and a diffeomorphism $\phi:\Uman\to\bR^{m}$ such that $\phi(x)=0$, and for every leaf $\omega$ with $\omega\cap \Uman \not=\varnothing$ and every connected component $\Kman$ of $\omega\cap\Uman$ its image $\phi(\Kman)$ coincides with the plane $\{v_1 = c\}$ for some $c\in\bR$.
But this plane is a connected component of a certain level set of the function $\func:\bR^m \to\bR$ defined by $\func(v_1,\ldots,v_m) = v_1^2$.
Hence condition~\ref{enum:MBfol:local_normal_form} holds, and so $\Fol$ is also a Morse-Bott foliation with singularities $\{\Cman\} \cup \{\Cman_i\}_{i\in\Lambda}$.
\end{proof}

It is proved in~\cite[Theorem~B]{ScarduaSeade:JDG:2009} that if $\Fol$ is a foliation with Morse-Bott singularities, and a singular leaf $\Cman_i$ has a finite holonomy group (e.g.\! finite fundamental group), then one can find an open neighborhood $\Nman_i$ of $\Cman_i$ and a Morse-Bott function $\func:\Nman_i\to\bR$ such that $\Cman_i$ is the set of all critical points of $\func$, and the restriction of $\Fol$ to $\Nman_i\setminus\Cman_i$ consists of connected components of sets $\func^{-1}(c) \setminus \Cman_i$, $c\in\bR$.

\begin{theorem}\label{th:MBfol:ltf}
Let $\Fol$ be a Morse-Bott foliation on a smooth manifold $\Mman$, $\Cman_1,\ldots,\Cman_k$ be any finite collection of its \myemph{compact} singular leaves, and $\Sigma = \mathop{\cup}\limits_{i=1}^{k}\Cman_i$.
Then for every closed (possibly empty) subset $\Xman\subset\Mman \setminus \Sigma$ the ``restriction to $\Sigma$'' homomorphisms
\begin{align*}
&\rho: \Diff(\Fol,\Xman) \to \Diff(\Sigma), &
\rho(\dif) &= \dif|_{\Sigma}, \\
&\rho_{0}: \DiffId(\Fol,\Xman) \to \prod_{i=1}^{k} \DiffId(\Cman_i), &
\rho_{0}(\dif) &= (\dif|_{\Cman_1}, \ldots, \dif|_{\Cman_k}),
\end{align*}
are locally trivial fibrations over their images and $\rho_0$ is surjective.
\end{theorem}

This theorem is a particular case of Theorem~\ref{th:Zfol:ltf} which also implies Theorems~\ref{th:MB_func:ltf} and~\ref{th:func_ltf}.

\section{Ehresmann connections}\label{sect:ehresmann_conn}
Let $p:\Eman\to\CCman$ be a smooth fibre bundle over a smooth connected manifold $\CCman$.
For $x\in\CCman$ we will denote by $\Eman_x = p^{-1}(x)$ the fibre over $x$.
Let
\[
\VE = \mathop{\cup}\limits_{(x,v)\in\Eman} T_{(x,v)}\Eman_x
\]
be the union of all tangent planes to the fibres.
The projection $q:\VE \to \Eman$ is called the \myemph{vertical subbundle} of $T\Eman$.

More generally, for an open subset $\aSet \subset \Eman$ let
\[
\VU=q^{-1}(\aSet) = T\aSet \cap \VE
\]
be the union of tangent spaces to fibres at points of $\aSet$.
Then a \myemph{local Ehresmann connection over $\aSet$} is a smooth vector bundle projection $\conn:T\aSet\to \VU$
over the identity map of $\aSet$ i.e.\ it makes commutative the following diagram:
\[
\xymatrix{
\VU \ \ar[rd]_{q} \ar@{^(->}[r]^{i} \ar@/^2pc/[rr]^{\id_{\VU}} & \ T\aSet \ \ar[r]^{\conn} \ar[d]_{\pi} & \ \VU \ar[ld]^{q} \\
& \aSet
}
\]
Thus for each $(x,v)\in\aSet$ the restriction of $\conn$ to the tangent space $T_{(x,v)}\aSet$ induces a linear projection $\conn_{(x,v)}:T_{(x,v)}\aSet \to T_{(x,v)} E_{x}$.

A local Ehresmann connection over $\Eman$ is called an \myemph{Ehresmann connection}.

The kernel $\ker(\conn)$ of $\conn$ is the subbundle of $T\aSet$ called \myemph{horizontal}.
Then we have a direct sum splitting $T\aSet = \HU \oplus \VU$.

If $\Nman$ is a smooth manifold, then a smooth map $\phi:\Nman\to\aSet$ is called \myemph{horizontal} with respect to $\conn$, or simply \myemph{$\conn$-horizontal}, if the image of the tangent bundle $T\Nman$ under the tangent map $T\phi:T\Nman\to T\aSet$ is contained in the horizontal subundle $\HU$, i.e.
\[T\phi(T\Nman) \subset \HU.\]

It is well known, e.g.~\cite[Theorem~9.8]{KolarMichorSlovak:NODG:1993}, that $\conn$ induces the so-called \myemph{parallel transport}.
Namely, if $\gamma:[0,1]\to\CCman$ is a smooth path and $z\in \Eman_{\gamma(0)} \cap \aSet$, then there exists an $\eps\in(0,1]$ and a unique smooth path
$\lift{\gamma}{z}:[0,\eps] \to \aSet$
such that
\begin{enumerate}[label={\rm(\roman*)}]
\item\label{enum:lift:a} $\lift{\gamma}{z}(0) = z$;
\item\label{enum:lift:b} $p\circ \lift{\gamma}{z} = \gamma:[0,\eps] \to \CCman$;
\item\label{enum:lift:c} $\lift{\gamma}{z}$ is $\conn$-horizontal.
\end{enumerate}
This path $\lift{\gamma}{z}$ is called a \myemph{horizontal lift} of $\gamma$ at $z$.
Moreover, $\lift{\gamma}{z}$ smoothly depends on $z$ and $\gamma$ in the sense described in Lemma~\ref{lm:local_homotopy_ext} below.
If $\aSet=\Eman$ and one can choose $\eps=1$ for any such path $\gamma$ and $z\in\Eman_{\gamma(0)}$, then the connection $\conn$ is called \myemph{complete}.

Let $\bSet \subset \aSet$ be an open subset and $\alpha = \{\Aman_i\}_{i\in\Lambda}$ be an open cover of $\CCman$.
We will say that a local Ehresmann connection $\conn:T\aSet\to\VU$ over $\aSet$ is \myemph{complete with respect to $(\bSet,\alpha)$} if for each smooth path $\gamma:[0,1]\to\CCman$ whose image is contained in some element of $\Aman_i$ of $\alpha$ and each $z\in \Eman_{\gamma(0)} \cap \bSet$ there exists a horizontal lift $\lift{\gamma}{z}:[0,1]\to \aSet$ satisfying the above conditions~\ref{enum:lift:a}-\ref{enum:lift:c}.

Thus a complete Ehresmann connection is a local Ehresmann connection over $\Eman$ being complete with respect to the pair $(\Eman, \{\CCman\})$.

\begin{lemma}\label{lm:linear_retr}
Let $E$ be a linear space over $\bR$, $V\subset E$ be a subspace, and $p:E\to E$ be a linear self-map with $p(E)=V$.
Then the following conditions are equivalent:
\begin{enumerate}[label={\rm(\roman*)}]
\item\label{enum:p_proj} $p$ is a \myemph{projection} onto its image $V$, that is $p^2=p$;
\item\label{enum:p_retr} $p$ is a \myemph{retraction} onto $V$, that is $p(v)=v$ for all $v\in V$.
\end{enumerate}

Moreover, let $p_0,p_1:E\to V$ be two projections onto $V$, and $p_t:E \to V$, $t\in[0,1]$, be given by $p_t = (1-t)p_0 + t p_1$.
Then
\begin{enumerate}[label={\rm(\alph*)}]
\item\label{lm:lin_retr:a} each $p_t$ is a projection onto $V$ as well;
\item\label{lm:lin_retr:b} $\ker p_t$ is contained in a convex hull of $\ker p_0$ and $\ker p_1$.
\end{enumerate}
\end{lemma}
\begin{proof}
The equivalence of conditions \ref{enum:p_proj} and \ref{enum:p_retr} is evident: each of them requires that $p(p(e))=p(e)$ for all $e\in E$.

\ref{lm:lin_retr:a}
If $v\in V$, then $p_t(v) = (1-t)p_0(v) + t p_1(v)=v$, and so $p_t$ is a linear retraction, and hence a projection, onto $V$.

\ref{lm:lin_retr:b}
Let $a\in \ker p_t$, so $(1-t)p_0(a) + t p_1(a)=0$.
Then $a_i = a-p_i(a) \in \ker p_i$ for $i=0,1$.
Indeed, as $p_i$ is a retraction, $p_i(a_i) = p_i(a)- p_i(p_i(a_i))= p_i(a)- p_i(a_i)=0$.

Hence,
\begin{align*}
    (1-t)a_0 + ta_1 &=
    (1-t)(a-p_0(a)) + t(a-p_1(a)) = \\
    & =(1-t) a + t a  +     (1-t) p_0(a) + tp_1(a) = a,
\end{align*}
that is $a$ belongs to the convex hull of $\ker p_0$ and $\ker p_1$.
\end{proof}



As a consequence of this lemma we get the following statement.
\begin{lemma}\label{lm:Ehresmann_conn_convex}
{\rm cf. \cite[Exercise 12(iv)]{GreubHalperinVanstone:I:1972}.}
The set of all Ehresmann connections $T\aSet \to \VU$ over an open subset $\aSet$ of $\Eman$ is convex in the linear space of all bundle morphisms of $T\aSet \to T\Eman$ over the identity inclusion $\aSet\subset \Eman$.

Moreover, let $\mathcal{L}= \mathop{\cup}\limits_{z\in \aSet} \mathcal{L}_z$ be a distribution on $\aSet$, i.e.\! a correspondence associating to each $z\in\aSet$ some linear subspace $\mathcal{L}_z$ of the tangent space $T_{z}\aSet$, (the dimensions of $\mathcal{L}_z$ may vary with $z$, and we do not assume any kind of continuity of the correspondence $z\mapsto \mathcal{L}_z$).
Denote by $\Gamma_{\aSet}(\mathcal{L})$ the set of all local Ehresmann connections over $\aSet$ whose horizontal bundles are contained in $\mathcal{L}$.
Then $\Gamma_{\aSet}(\mathcal{L})$ is convex as well.
\end{lemma}
\begin{proof}
Let $\conn^0,\conn^1: T\aSet \to \VU$ be two local Ehresmann connections over some open subset $\aSet \subset \Eman$, so for each $(x,v)\in\aSet$ their restrictions to the tangent space $T_{(x,v)}\aSet$ are well-defined projections $\conn^0_{(x,v)},\conn^1_{(x,v)}: T_{(x,v)}\aSet \to  T_{(x,v)}\Eman_x.$ 
Then for each $t\in[0,1]$ we have a well-defined bundle morphism 
\[ \conn^t = (1-t)\conn^0 + t\conn^1:T\aSet \to \VU,\] 
inducing for each $(x,v)\in\aSet$ a linear map $\conn^t_{(x,v)}:T_{(x,v)}\aSet \to  T_{(x,v)}\Eman_x$.
Since 
\[\conn^t_{(x,v)}=(1-t)\conn^0_{(x,v)} + t\conn^1_{(x,v)}\] is a convex linear combination of linear projections $\conn^0_{(x,v)}$ and $\conn^1_{(x,v)}$, it follows from~\ref{lm:lin_retr:a} of Lemma~\ref{lm:linear_retr} that $\conn^t_{(x,v)}$ is a linear projection as well.
Thus $\conn^t$ is a local Ehresmann connection over $\aSet$, and thus the set of all Ehresmann connections $T\aSet \to \VU$ over $\aSet$ is convex.

Moreover, let $\mathcal{L}$ be a distribution on $\aSet$ such that $\HBundle{\conn_0}, \HBundle{\conn_1} \subset \mathcal{L}$.
In other words, for each $(x,v)\in \aSet$, the kernels of the linear projections $\conn^0_{(x,v)}$ and $\conn^1_{(x,v)}$ are contained in the corresponding plane $\mathcal{L}_{(x,v)} \subset T_{(x,v)} \aSet$.
Then by~\ref{lm:lin_retr:b} of Lemma~\ref{lm:linear_retr} the kernel of $\conn^t_{(x,v)}$ is contained in the convex hull of kernels of $\conn^0_{(x,v)}$ and $\conn^1_{(x,v)}$, and therefore it must also be contained in $\mathcal{L}_{(x,v)}$.
Thus $\conn^t \in \Gamma_{\aSet}(\mathcal{L})$ for all $t\in[0,1]$ and so $\Gamma_{\aSet}(\mathcal{L})$ is convex.
\end{proof}
Using a partition of unity technique one easily deduces from Lemma~\ref{lm:Ehresmann_conn_convex} that every smooth fibre bundle $p:\Eman\to\CCman$ always admits an Ehresmann connection.
However, as it is mentioned in~\cite{Hoyo:IndagMath:2016} the subspace of \myemph{complete} Ehresmann connections is not necessarily convex.
Hence gluing complete Ehresmann connections defined in local charts gives an Ehresmann connection which is not necessarily complete.
Nevertheless, it is shown in~\cite{Hoyo:IndagMath:2016} that every smooth fibre bundle admits a complete Ehresmann connection.
For our purposes it is enough to use only local Ehresmann connections being complete with respect to some pair $(\bSet,\alpha)$, see the last paragraph of the proof of Theorem~\ref{th:loc_triv_fibr}, and we will not use the result of \cite{Hoyo:IndagMath:2016}.

\begin{lemma}[Homotopy lifting property]\label{lm:local_homotopy_ext}{\rm cf. \cite[Exercise 13(i)]{GreubHalperinVanstone:I:1972}}.
Let $\conn:T\aSet \to \VU$ be a local Ehresmann connection complete with respect to an open set $\Zman \subset \aSet$ and an open cover $\alpha=\{\Aman_i\}_{i\in\Lambda}$ of $\CCman$.
Let also $\Wman$ be a smooth manifold and $\aHmt: \Wman\times [0,1]\to \CCman$ be a smooth homotopy.
Suppose there exists a smooth map $\bHmt_0:\Wman\times 0 \to \bSet$ such that $p \circ \bHmt_0 = \aHmt_0$.
Then there exists an open neighborhood $\Gman$ of $\Wman\times 0$ in $\Wman\times [0,1]$ and a unique smooth map $\bHmt: \Gman \to \aSet$ having the following properties:
\begin{enumerate}[label={\rm(\alph*)}]
\item\label{enum:lm:local_homotopy_ext:H0_id} $\bHmt(w,0) = \bHmt_0(w)$ for all $w\in\Gman$;
\item\label{enum:lm:local_homotopy_ext:phH_H} $p\circ \bHmt = \aHmt|_{\Gman}$;
\item\label{enum:lm:local_homotopy_ext:hor} if $w\times[0,\eps] \subset \Gman$ for some $w\in\Gman$ and $\eps>0$, then the path $\bHmt_{w}: [0,\eps]\to \Eman$ defined by $\bHmt_{w}(t) = \bHmt(w,t)$ is $\conn$-horizontal.
\end{enumerate}
In particular, we have the following commutative diagram:
\[
\xymatrix{
\ \Wman\times 0 \ \ar@{^(->}[d] \ar[r]^-{\bHmt_0} & \ \bSet \ \ar@{^(->}[r] & \ \aSet \ \ar@{^(->}[r] & \ E \ar[d]^{p} \ \\
\ \Gman \ \ar@{^(->}[r] \ar@{-->}[rru]^-{\bHmt} & \ \Wman\times [0,1] \ \ar[rr]^-{\aHmt}          & & \ \CCman \
}
\]
If $\conn$ is a complete Ehresmann connection then one can take $\Gman = \Wman\times[0,1]$.
\end{lemma}
\begin{proof}
Conditions~\ref{enum:lm:local_homotopy_ext:H0_id}-\ref{enum:lm:local_homotopy_ext:hor} uniquely determine the map $\bHmt$.
Indeed, let $z=(x,v)\in\Wman$ and $\gamma_{z}:[0,1]\to\CCman$ be given by $\gamma_z(t) = \aHmt(z,t)$.
Let $I_z \subset [0,1]$ be a maximal subinterval such that $\gamma_{z}(I_z)$ is contained in some element $\Aman_i$ of $\alpha$ and on which a unique horizontal lift $\lift{\gamma}{z}:I_z \to \aSet$ of $\gamma$ is defined.
Put $\Gman' = \mathop{\cup}\limits_{z\in\Wman} z\times I_z$ and define the following map $\bHmt:\Gman' \to \aSet$ by $\bHmt(z,t) = \lift{\gamma}{z}(t)$.

Notice that $\lift{\gamma}{z}$ is a solution of a certain ODE whose coefficients are smooth functions of derivatives of $\gamma$.
This implies existence and uniqueness of the lift $\lift{\gamma}{z}$, its smooth dependence on the initial point $z$, and also that $\Gman'$ contains an open neighborhood $\Gman$ of $\Wman\times 0$.
Then the restriction $\bHmt|_{\Gman}: \Gman \to \aSet$ is the required map.
We leave the details for the reader.
\end{proof}

\section{Local triviality of restriction maps to singluar leaves}\label{sect:loc_triv_emb}
Let $p:\Eman\to\CCman$ be a smooth $k$-dimensional vector bundle over a smooth manifold $\CCman$.
We will identify $\CCman$ with its image in $\Eman$ under the zero section.

Let $\Cf{\Eman}{\CCman}$ be the subset of $\Ci{\Eman}{\Eman}$ consisting maps $\dif:\Eman\to\Eman$ such that $\dif(\CCman)=\CCman$.
For an open $\aSet\subset \Eman$ denote by $\CfU{\aSet}{\Eman}{\CCman}$ the subset of $\Cf{\Eman}{\CCman}$ consisting of maps fixed on $\Eman\setminus\aSet$.
Let also $\Diff(\Eman)$ be the group of $\Cinfty$ diffeomorphisms of $\Eman$ and
\begin{align*}
\Df{\Eman}{\CCman} &= \Cf{\Eman}{\CCman} \cap \Diff(\Eman), &
\DfU{\aSet}{\Eman}{\CCman} &= \CfU{\aSet}{\Eman}{\CCman} \cap \Diff(\Eman).
\end{align*}
We endow all these spaces with the corresponding strong Whitney $\Cinfty$ topologies.
Since $\Diff(\Eman)$ is open in $\Ci{\Eman}{\Eman}$, it follows that $\Df{\Eman}{\CCman}$ is open in $\Cf{\Eman}{\CCman}$ and $\DfU{\aSet}{\Eman}{\CCman}$ is open in $\CfU{\aSet}{\Eman}{\CCman}$.

Notice that there is a natural ``restriction to $\CCman$ homomorphism'':
\[
\rho: \Df{\Eman}{\CCman} \to \Diff(\CCman),
\qquad
\rho(\dif) = \dif|_{\CCman}.
\]
It is well known that $\rho$ is a locally trivial fibration for compact $\CCman$, \cite{Palais:CMH:1960}, \cite{Cerf:BSMF:1961}, \cite{Lima:CMH:1964}.
We present now a slight modification of the proof from E.~Lima~\cite{Lima:CMH:1964} which uses local Ehresmann connections and hold for non-compact $\CCman$ as well.
This will allow us to establish also a local triviality of restrictions to singular leaves for certain groups of leaf preserving diffeomorphisms, see Theorem~\ref{th:Zfol:ltf}.

\begin{remark}\rm
Let $p:\Eman\to\CCman$ be a smooth fibre bundle.
Chapter~3 of \cite{HongKalliongisMcCulloughRubinstein:LMN:2012} extends Cerf-Palais result about local triviality for embeddings ``in a vertical direction'', that is for submanifolds $\Xman\subset\Eman$ being unions of fibres and diffeomorphisms preserving fibres of vertical subbundle.
\end{remark}

\begin{theorem}\label{th:loc_triv_fibr}{\rm cf.~\cite{Palais:CMH:1960}, \cite{Cerf:BSMF:1961}, \cite{Lima:CMH:1964}.}
Let $\CCman$ be a closed manifold, $p:\Eman\to\CCman$ be a smooth vector bundle over $\CCman$, and $\aSet$ be an open neighborhood of $\CCman$ in $\Eman$ with compact closure.
Then for every subgroup $\mathcal{A} \subset \Df{\Eman}{\CCman}$ containing $\DfU{\aSet}{\Eman}{\CCman}$ the ``restriction to $\CCman$ homomorphism'':
\begin{equation}\label{equ:restr_A_DfM}
\rho: \mathcal{A} \to \Diff(\CCman),
\qquad\qquad
\rho(\dif) = \dif|_{\CCman}
\end{equation}
is a locally trivial fibration over its image.
\end{theorem}
\begin{proof}
We need the following three lemmas.
The first one is well known, but for completeness we present a short proof.
\begin{sublemma}\label{lm:sect_hom}{\rm cf. \cite[Theorem~A]{Palais:CMH:1960}}
Let $\rho: A \to B$ be a continuous homomorphisms of topological groups with units $e_{A}$ and $e_{B}$ respectively.
Suppose there exists an open neighborhood $\Uman$ of $e_{B}$ in $B$ and a continuous map $\msect:\Uman \to A$ such that
\begin{enumerate}[itemsep=0.5ex, label={\rm(\alph*)}]
\item\label{enum:lm:sect_hom:sieB_eA}
$\msect(e_{B})=e_{A}$;
\item\label{enum:lm:sect_hom:risig_g}
$\rho\circ\msect(g) = g$ for all $g\in \Uman$.
\end{enumerate}
Then the following statements hold.
\begin{enumerate}[label={\rm\arabic*)}]
\item\label{enum:lm:sect_hom:x:1}
The induced surjective homomorphism $\rho:A \to \rho(A)$ is locally trivial fibration with fiber $\ker(\rho)$.

\item\label{enum:lm:sect_hom:x:2}
The image of $\rho$ consists of entire path components of $B$.
In particular, if $A_e$ and $B_e$ are the path components of the units in $A$ and $B$ respectively, then $\rho(A_e)=B_e$.

\item\label{enum:lm:sect_hom:x:3}
If $\msect(\Uman\cap B_e)\subset A_e$, then the restriction map $\rho|_{A_e}:A_e\to B_e$ is a locally trivial fibration with fibre $\ker(\rho)\cap A_e$.
\end{enumerate}
\end{sublemma}
\begin{proof}
\ref{enum:lm:sect_hom:x:1}
First we will check that $\rho$ is a locally trivial fibration.
Let $a\in A$ and $b=\rho(a) \in B$.
Then by~\ref{enum:lm:sect_hom:sieB_eA} the set $\Uman_{b} = b\Uman$ is an open neighborhood of $b$ in $B$.
Define the following map $\phi: \Uman_{b} \times \ker(\rho) \to \rho^{-1}(\Uman_{b})$ by
\[\phi(g,k) = a\, \xi(b^{-1}g)\, k.\]
Then $\phi$ is a homeomorphism making commutative the following diagram:
\[
\xymatrix{
\Uman_{b}\times\ker(\rho) \ar[rr]^-{\phi} \ar[rd] && \rho^{-1}(\Uman_{b}) \ar[ld]_-{\rho} \\
& \Uman_{b}
}
\]
and its inverse $\phi^{-1}: \rho^{-1}(\Uman_{b}) \to \Uman_{b} \times \ker(\rho)$ is given by
\[
\phi^{-1}(x) = ( \rho(x), \xi(b^{-1}\rho(x))^{-1} a^{-1} x ).
\]
Indeed, notice that
\[
\rho(\phi(g,k)) = \rho(a\, \xi(b^{-1}g)\, k) = \rho(a) \, \rho(\xi(b^{-1}g))\, \rho(k) = b b^{-1} g = g.
\]
Hence
\begin{align*}
\phi^{-1}\circ\phi(g,k) &=
\phi^{-1}(a\, \xi(b^{-1}g)\, k) =
\bigl( g, \xi(b^{-1} g)^{-1} a^{-1} \, a\, \xi(b^{-1}g)\, k \bigr) = (g,k), \\
\phi\circ\phi^{-1}(x) &=
\phi( \rho(x), \xi(b^{-1}\rho(x))^{-1} a^{-1} x ) =
a \xi( b^{-1}\rho(x) ) \xi(b^{-1}\rho(x))^{-1} a^{-1} x = x.
\end{align*}
Thus $\rho:A \to B$ is a locally trivial fibration.

\ref{enum:lm:sect_hom:x:2}
Let us prove that $\rho(A)$ consists of entire path components.
Let $b = \rho(a) \in \rho(A)$ and $\gamma:[0,1]\to B$ be a path such that $\gamma(0)=b$ and $\gamma(1)=b'$.
We will show that $b' = \rho(a')$ for some $a'\in A$.
Indeed, since $\rho$ is a locally trivial fibration (and thus has the path lifting property), $\gamma$ lifts to a path $\tilde{\gamma}:[0,1]\to A$ such that $\tilde{\gamma}(0)=a$ and $\rho\circ\tilde{\gamma} = \gamma$.
Denote $a'=\tilde{\gamma}(1)$.
Then $b' = \gamma(1) = \rho\circ\tilde{\gamma}(1) = \rho(a')$.

\ref{enum:lm:sect_hom:x:3}
Finally notice that $\Vman = \Uman\cap B_{e}$ is an open neighborhood of $e_{B}$ in $B_{e}$.
Hence if $\msect(\Uman\cap B_{e}) \subset A_e$, then by 1) the restriction $\rho|_{A_e}$ is a locally trivial fibration as well.
\end{proof}

\begin{sublemma}\label{lm:existence_complete_part_conn}
Let $p:\Eman\to\CCman$ be a smooth vector bundle over a not necessarily compact manifold $\CCman$, and $\conn:T\aSet \to \VU$ be a local Ehresmann connection over an open neighborhood $\aSet\subset \Eman$ of $\CCman$.
Suppose $T\CCman\subset \HU$.
Then for each open cover $\beta$ of $\CCman$ there exists an open neighborhood $\bSet \subset \aSet$ of $\CCman$ and an open cover $\alpha=\{\Aman_i\}$ of $\CCman$ such that
\begin{enumerate}[label={\rm(\roman*)}]
\item\label{enum:lm:existence_complete_part_conn:a} $\alpha$ is a refinement of $\beta$, that is each element of $\alpha$ is contained in some element of $\beta$;
\item\label{enum:lm:existence_complete_part_conn:b} the closures of $\Aman_i$ are compact and constitute a locally finite family;
\item\label{enum:lm:existence_complete_part_conn:c} $\conn$ is complete with respect to $(\bSet,\alpha)$.
\end{enumerate}
\end{sublemma}
\begin{proof}
Fix a complete Riemannian metric of $\Eman$ and for each $x\in\CCman$ and $\delta>0$ denote by $\ball{x}{\delta}$ the open geodesic ball of radius $\delta$ in the fibre $\Eman_{x}$ centered at its origin which is identified with the point $x$ itself.
More generally, for a subset $\Xman\subset \CCman$ and $\delta>0$ denote $\ball{\Xman}{\delta} = \mathop{\cup}\limits_{x\in\Xman}\ball{x}{\delta}$.

Let $\gamma:[0,1]\to\CCman \subset \Eman$ be a smooth path and $x = \gamma(0)$.
Since $x\in\CCman$ and $T\CCman\subset\HU$, it follows that the horizontal lift $\lift{\gamma}{x}$ of $\gamma$ coincides with $\gamma$.
Hence we get from continuous dependence of solutions of ODE on initial data that for each $x\in\CCman$ there exists an open neighborhood $\Aman_x\subset\CCman$ with compact closure $\overline{\Aman_x}$ and a number $\delta_x>0$ such that
\begin{itemize}
\item $\{ \Aman_x \}_{x\in\CCman}$ is a refinement of $\beta$;
\item for every smooth path $\gamma:[0,1]\to\Aman_x$ with $\gamma(0)=x$ and $z\in\ball{x}{\delta_x}$ there exists a horizontal lift $\lift{\gamma}{z}:[0,1]\to\aSet$.
\end{itemize}
As $\CCman$ is paracompact and has a countable base, one can find at most countalble locally finite refinement $\alpha=\{ \Aman_i \}_{i\in\Lambda \subseteq\bN}$ of $\{ \Aman_x \}_{x\in\CCman}$.
Then conditions~\ref{enum:lm:existence_complete_part_conn:a} and~\ref{enum:lm:existence_complete_part_conn:b} hold for $\alpha$.
Moreover, as each $\Aman_i \in \alpha$ intersects only finitely many other elements of $\alpha$, the following number is positive:
\[
\eta_i = \min\{  \delta_j \mid \Aman_j \cap \Aman_i \not=\varnothing, \ j\in\Lambda\} > 0.
\]
Put
\[
\bSet = \mathop{\cup}\limits_{i\in\Lambda} \ball{\Aman_i}{\eta_i}.
\]
Then one easily checks that $\conn$ is complete with respect to the pair $(\bSet,\alpha)$ as well.
\end{proof}

\begin{sublemma}\label{lm:sect_at_id}
{\rm\cite{Lima:CMH:1964}}
Let $p:\Eman\to\CCman$ be a smooth vector bundle over a closed manifold $\CCman$, and $\conn:T\aSet \to \VU$ be a local Ehresmann connection over an open neighborhood $\aSet\subset \Eman$ of $\CCman$ having compact closure and such that $T\CCman\subset \HU$.
Then there exists an open neighborhood $\VV$ of $\id_{\CCman}$ in $\Diff(\CCman)$ and a homotopy $\msect:\VV\times[0,1]\to \DfU{\aSet}{\Eman}{\CCman}$ such that
\begin{enumerate}[label={\rm(\alph*)}, topsep=1ex, itemsep=0.5ex]
\item\label{enum:lm:sect_at_id:F0_pt}
$\msect_0(\VV) = \id_{\Eman}$, that is $\msect_0(\gdif)(x,v) = (x,v)$ for all $\gdif\in\VV$;

\item\label{enum:lm:sect_at_id:loc_sect_adopted}
for each $\gdif\in\VV$ and $(x,v)\in\aSet$ the path $t\mapsto\msect_t(\gdif)(x,v)$ is horizontal;

\item\label{enum:lm:sect_at_id:xi_q__h}
$\msect_t(\id_{\CCman}) = \id_{\Eman}$ for all $t\in[0,1]$;

\item\label{enum:lm:sect_at_id:rho_xi__idU}
$\rho\circ\msect_1 = \id_{\VV}$, that is $\msect_1(\gdif)|_{\CCman} = \gdif$ for all $\gdif\in\VV$.
\end{enumerate}
In particular, $\msect_1$ is a section of the restriction map $\rho$.
\end{sublemma}
\begin{proof}
The proof follows the line of \cite{Lima:CMH:1964}.
By Lemma~\ref{lm:existence_complete_part_conn} there exists an open neighborhood $\bSet \subset\aSet$ of $\CCman$ and a finite (as $\CCman$ is compact) open cover $\alpha = \{\Aman_i\}_{i=1,\ldots,a}$ of $\CCman$ such that $\conn$ is complete with respect to $(\bSet, \alpha)$.

Fix an embedding of $\CCman$ into $\bR^n$ and let $\pi:\Tman\to\CCman$ be its tubular neighborhood.
Then there exists $\tau>0$ such that for every two points $x,y\in\CCman$ with $|y-x|<\tau(x)$ the segment $I_{x,y}$ between $x$ and $y$ is contained in $\Tman$, and its projection $\pi(I_{x,y})$ is contained in some element $\Aman_i$ of the cover $\alpha$.

Denote by $\WW$ the set of all $\gdif\in\Diff(\CCman)$ such that $|\gdif(x) - x| < \tau$ for each $x\in\CCman$.
Then $\WW$ is an open neighborhood of $\id_{\CCman}$ in $\Diff(\CCman)$.
We will construct a map
\[\msect:\WW \times[0,1] \to \CfU{\aSet}{\Eman}{\CCman}\]
satisfying~\ref{enum:lm:sect_at_id:F0_pt}-\ref{enum:lm:sect_at_id:rho_xi__idU}, and then choose a smaller neighborhood $\VV \subset\WW$ of $\id_{\CCman}$ such that
\[\msect(\VV\times[0,1]) \subset \DfU{\aSet}{\Eman}{\CCman}.\]

\medskip

Fix a $\Cinfty$ function $\lambda:\Eman\to[0,1]$ such that $\lambda=0$ on some neighborhood $\Nman$ of $\overline{\Eman\setminus\Zman}$ and $\lambda=1$ near $\CCman$.
For each $\gdif\in\WW$ consider the following map:
\begin{align*}
&\aHmt^{\gdif}:\Eman\times[0,1]\to\CCman, &
\aHmt^{\gdif}(x,v,t) &= \pi\bigl( x + t \lambda(x, tv) (\gdif(x) - x)\bigr),
\end{align*}
where $x\in\CCman$, $v\in\Eman_{x}$, and $t\in[0,1]$.
Then
\[
\aHmt^{\gdif}(x,v,0) = \pi(x) = x = p(x,v).
\]
Define also the map $\bHmt^{\gdif}_0=\id_{\Eman}:\Eman\times 0 \to \Eman$ by $\aHmt^{\gdif}_0(x,v)=(x,v)$.
Then we have the following commutative diagram:
\[
\xymatrix{
\Eman\times 0 \ar@{^(->}[d] \ar[rr]^-{\bHmt^{\gdif}_0 = \id_{\Eman}} && E \ar[d]^-{p} \\
\Eman\times [0,1] \ar@{-->}[rru]^-{\bHmt^{\gdif}} \ar[rr]^-{\aHmt^{\gdif}} && \CCman
}
\]
We want to construct a diagonal map $\bHmt^{\gdif}:\Eman\times [0,1]\to\Eman$ extending $\bHmt^{\gdif}_0$ and satisfying $\aHmt^{\gdif} = p \circ \bHmt^{\gdif}$.

By assumption for each $x\in\CCman$ the path $\pi(I_{x,\gdif(x)})$ is contained in some element $\Aman_i$ of the cover $\alpha$.
Since $\conn$ is complete with respect to the pair $(\bSet,\alpha)$, we get from Lemma~\ref{lm:local_homotopy_ext} that there exists a unique smooth lift $\bHmt^{\gdif}:\bSet\times[0,1]\to\Eman$ such that $\aHmt^{\gdif} = p\circ \bHmt^{\gdif}$ and for every $(x,v)\in\bSet$ the path $t\mapsto\bHmt^{\gdif}(x,v,t)$ is horizontal.
Extend $\bHmt^{\gdif}$ to the map $\bHmt^{\gdif}:\Eman\times[0,1]\to\Eman$ by $\bHmt^{\gdif}(z,t) = z$ for $z\in\Eman\setminus\bSet$.

We claim that $\bHmt^{\gdif}$ is smooth.
It suffices to show that $\bHmt^{\gdif}(z,t) = z$ for $z\in\bSet \cap \Nman$.

Let $(x,v) \in \bSet\cap \Nman$.
Then $\lambda(x,v) = 0$, whence
\[
\aHmt^{\gdif}(x,v,t) = \pi\bigl( x + \lambda(x,v) (\gdif(x) - x)\bigr) = \pi (x) = x
\]
for all $t\in[0,1]$.
In other words, $t\mapsto \aHmt^{\gdif}(x,v,t)$ is a constant path, whence its lifting $t\mapsto \bHmt^{\gdif}(x,v,t)$ is also constant, that is $\bHmt^{\gdif}(x,v,t) = \bHmt^{\gdif}(x,v,0) = (x,v)$.

Thus each $\bHmt^{\gdif}_t$ is supported in $\bSet$, so in particular, we have a well-defined map
\begin{align*}
&\msect: \WW\times[0,1] \to \CfU{\aSet}{\Eman}{\CCman}, &
\msect(\gdif, t)(x,v) &= \bHmt^{\gdif}(x,v,t).
\end{align*}
One easily checks that $\msect$ is continuous.
Let us verify properties~\ref{enum:lm:sect_at_id:F0_pt}-\ref{enum:lm:sect_at_id:rho_xi__idU} for $\msect$.
Evidently, \ref{enum:lm:sect_at_id:F0_pt} and \ref{enum:lm:sect_at_id:loc_sect_adopted} follow from the definition.

\smallskip

\ref{enum:lm:sect_at_id:xi_q__h}
We should prove that $\msect_t(\id_{\CCman}) = \id_{\Eman}$ for all $t\in[0,1]$.
By the construction $\msect_t(\gdif)$ is fixed on $\Eman\setminus\bSet$ for all $\gdif\in\WW$ and $t\in[0,1]$.
Therefore it suffices to prove that $\msect_t(\id_{\CCman})(x,v) = (x,v)$ for $(x,v)\in\bSet$.
Notice that
\[\aHmt^{\id_{\CCman}}(x,v,t) = \pi\bigl(x + t\lambda(x, tv) (x - x)\bigr) = \pi(x) = x\]
for all $t\in[0,1]$.
In other words, the path $t\mapsto \aHmt^{\id_{\CCman}}(x,v,t)$ is constant, and therefore so is its horizontal lift $t \mapsto \bHmt^{\id_{\CCman}}(x,v,t)$.
Hence
\[
\msect(\id_{\CCman})(x,v) = \bHmt^{\id_{\CCman}}(x,v,1) = \bHmt^{\id_{\CCman}}(x,v,0) = (x,v).
\]

\smallskip

\ref{enum:lm:sect_at_id:rho_xi__idU}
Recall that we identified $\CCman$ with the zero section of $p$.
Let $(x,0)\in\CCman \subset \Eman$, $\gdif\in\WW$, and $\gamma_x:[0,1]\to\CCman$ be given by $\gamma_x(t) = \aHmt^{\gdif}(x,0,t)$.
By assumption $\lambda(x,0)=1$ for all $x\in\CCman$, whence
\begin{align*}
\gamma_x(0) &= \bHmt^{\gdif}(x,0,0) = \pi(x) = x, \\
\gamma_x(1) &= \bHmt^{\gdif}(x,0,1) = \pi\bigl( x + \lambda(x, 0) (\gdif(x) - x)\bigr) = \pi(\gdif(x))=\gdif(x).
\end{align*}
The latter equality holds since $\gdif(x)\in\CCman$.
As $T\CCman\subset \HU$, it follows that the horizontal lift of $\gamma_x$ coincides with $\gamma_x$ itself.
In other words,
\[
\bHmt^{\gdif}(x,0,t) = \bigl( \aHmt^{\gdif}(x,0,t), 0 \bigr) \equiv  \aHmt^{\gdif}(x,0,t).
\]
In particular,
\[
\rho\circ\msect(\gdif)(x) =
\msect(\gdif)|_{\CCman}(x) = \msect(\gdif)(x,0) = \bHmt^{\gdif}(x,0,1)=
 \aHmt^{\gdif}(x,0,1) = \bigl(\gdif(x), 0\bigr) \equiv \gdif(x),
\]
which means that $\rho\circ\msect(\gdif) = \gdif$.

\medskip

Now let us show how to choose a neighborhood $\VV$.
Due to~\ref{enum:lm:sect_at_id:xi_q__h},
\[\msect(\id_{\CCman} \times [0,1]) = \id_{\Eman} \in \DfU{\aSet}{\Eman}{\CCman}.\]
Since the group $\DfU{\aSet}{\Eman}{\CCman}$ is open in $\CfU{\aSet}{\Eman}{\CCman}$, it follows that the set $\mathcal{Q} := \msect^{-1}(\DfU{\aSet}{\Eman}{\CCman})$ will be an open neighborhood of $\id_{\CCman} \times[0,1]$ in $\WW \times[0,1]$.
Moreover, as $[0,1]$ is compact, there exists an open neighborhood $\VV$ of $\id_{\CCman}$ in $\WW$ such that $\VV\times[0,1]\subset\mathcal{Q}$.
Then the restriction $\msect|_{\VV\times[0,1]}: \VV\times[0,1] \to \DfU{\aSet}{\Eman}{\CCman}$ will be a required homotopy.
\end{proof}

Now we can prove Theorem~\ref{th:loc_triv_fibr}.
Let $\conn$ be any Ehresmann connection on $\Eman$ such that $T\CCman\subset \HBundle{\conn}$.
The latter inclusion holds for instance for any linear connection.
Then by Lemma~\ref{lm:sect_at_id} the map $\rho$ defined by~\eqref{equ:restr_A_DfM} admits a local section on some neighborhood of $\id_{\CCman}$, whence by Lemma~\ref{lm:sect_hom} $\rho$ is a locally trivial fibration.
\end{proof}

\section{Fiberwise Morse lemma}\label{sect:fiber_wise_Morse}
The following lemma shows that the implication \ref{enum:MBF:Hess_non_deg}$\Rightarrow$\ref{enum:MBF:normal_form} in Definition~\ref{def:MB_map} can be regarded as a ``fiberwise'' analogue of Morse lemma.

\begin{lemma}\label{lm:fiberwise_morse_lemma}
Let $\Uman \subset\bR^{k}$ and $\Vman \subset \bR^{l}$ be open neighborhoods of the corresponding origins and $\func:\Uman\times\Vman\to\bR$ be a smooth function having the following properties:
\begin{enumerate}[label={\rm(\arabic*)}]
\item
$\func^{-1}(0)=\Uman\times 0$ and this set coincides with the set of critical points of $\func$;

\item
at each critical point $(u,0)\in\Uman\times 0$ the Hessian of $\func$ in the $\Vman$-direction
\[
H(f,u) =
\left(
\frac{\partial^2\func}{\partial v_i \partial v_j}(u,0)
\right)_{i,j=1,\ldots,l}
\]
is non-degenerate, where $v=(v_1,\ldots,v_l)$ are coordinates in $\bR^{l}$.
\end{enumerate}

Then there exist smaller neighborhoods $\hUman \subset \Uman$ and $\hVman\subset\Vman$ of the corresponding origins and a smooth embedding $\dif:\hUman\times\hVman \to \Uman\times\Vman$ such that
\begin{enumerate}[label={\rm(\alph*)}]
\item\label{enum:lm:fbm:fiberwise}
$\dif(u,v) = \bigl(u, \gdif(u,v)\bigr)$, $(u,v)\in\hUman\times\hVman$, for some smooth map $\gdif:\hUman\times\hVman \to \Vman$, i.e. \myemph{$\dif$ preserves the first coordinate $u$};
\item\label{enum:lm:fbm:normal_form}
$\func\circ \dif(u,v_1,\ldots,v_l) = - v_1^2 - \cdots - v_{\lambda}^2 + v_{\lambda+1}^2 + \cdots + v_{l}^2$ for some $\lambda$.
\end{enumerate}
\end{lemma}
\begin{proof}
The principal observation is that in the arguments of the standard proof of Morse lemma all changes of coordinates preserve coordinate $u$, see e.g.~\cite[Chapter~2, \S6]{GolubitskiGuillemin:StableMats:ENG:1973}.
This will guarantee condition~\ref{enum:lm:fbm:fiberwise}.
We leave the details for the reader and also refer to~\cite{BanyagaHurtubise:EM:2004} for several proofs of Morse-Bott lemma.
\end{proof}

\begin{lemma}\label{lm:bott_func_has_conn}
Let $p:\bR^{k+l} \equiv \bR^{k}\times \bR^{l}\to \bR^{k}$, $p(u,v) = u$, be a trivial vector bundle, $\lambda\in\{0,\ldots,l\}$, and $\func:\bR^{k+l} \to \bR$ be a Morse-Bott function defined by
\[
 \func(u,v_1,\ldots,v_l) = - v_1^2 - \cdots - v_{\lambda}^2 + v_{\lambda+1}^2 + \cdots + v_{l}^2,
\]
for $u\in\bR^{k}$ and $v = (v_1,\ldots,v_l) \in\bR^{l}$.
Then there exists an Ehresmann connection \[ \conn:T\bR^{k+l} \to \VBundle{\bR^{k+l}} \] whose horizontal bundle $\HU$ is contained in the distribution $\ker(d\func)$ of the kernels of the differential $d\func$ of $\func$.
\end{lemma}
\begin{proof}
Identify the tangent bundle $T\bR^{k+l}$ of $\bR^{k+l}$ with $\bR^{k}\times \bR^{l} \times \bR^{k}\times \bR^{l}$, so that the projection $q:T\bR^{k+l} \to \bR^{k+l}$ is given by $q(u,v,a,b) = (u,v)$, $u,a\in\bR^{k}$ and $v,b\in\bR^{l}$.
Then the corresponding vertical subbundle $\VBundle{\bR^{k+l}}$ is $\bR^{k+l} \times 0 \times \bR^{l}$, while the differential $d\func:T\bR^{k+l} \to \bR$ of $\func$ is given by the formula:
\[
    d\func(u,v, a,b) = - 2v_1 b_1 - \cdots - 2v_{\lambda} b_{\lambda} + 2v_{\lambda+1} b_{\lambda+1} + \cdots + 2v_{l} b_l,
\]
for $v=(v_1,\ldots,v_l)$ and $b=(b_1,\ldots,b_l) \in\bR^{l}$.
Define the following Ehresmann connection $\conn:T\bR^{k+l} \to \VBundle{\bR^{k+l}}$ by $\conn(u,v,a,b) = (u,v,0,b)$.
Then $\HU = \bR^{k+l} \times \bR^{k}\times 0$ and for each $(u,v,a,0)\in \HU$ we have that $d\func(u,v,a,0) = 0$, whence $\HU \subset \ker(d\func)$.
\end{proof}

\section{Leaf-preserving diffeomorphisms of singular foliations}\label{sect:diff_leaf_pres_fol}
\begin{definition}\label{def:Z:fol}
Let $\Mman$ be a smooth manifold, $\Fol$ be a partition of $\Mman$, and $\{\Cman_i\}_{i\in\Lambda}$ be a \myemph{discrete} family of elements of $\Fol$.
We say that $\Fol$ is a \myemph{foliation of class $\Eclass$} with singularities $\{\Cman_i\}_{i\in\Lambda}$ whenever
\begin{enumerate}[leftmargin=*, label={\rm(\alph*)}]
\item the restriction of $\Fol$ to $\Mman\setminus\mathop{\cup}\limits_{i\in\Lambda}\Cman_i$ is a smooth codimension one foliation, and every connected component of $\partial\Mman$ is a leaf of $\Fol$;

\item\label{enum:Z:sing_leaf_props}
every $\Cman_i$ is a smooth connected boundaryless submanifold of $\Mman$ and $\dim\Cman_i<\dim\Mman$;

\item\label{enum:Z:Ehresmann_conn}
for each $i\in\Lambda$ and there exists an open tubular neighborhood $\Nman_i$ of $\Cman_i$ in $\Mman$ and a retraction $p:\Nman_i\to\Cman_i$ having a structure of a vector bundle such that for each $x\in\Cman_i$ one can find a local Ehresmann connection $\conn_{x}$ over some open neighborhood $\aSet_x$ of $x$ in $\Nman_i$ with $\HBundle{\conn_{x}} \subset T\Fol$.
\end{enumerate}
\end{definition}

In that case each $\Cman_i$ will be called a \myemph{singluar} leaf of $\Fol$.
Lemmas~\ref{lm:fiberwise_morse_lemma} and~\ref{lm:bott_func_has_conn} imply that every Morse-Bott foliation, see Definition~\ref{def:MB_fol}, belongs to class $\Eclass$.
Therefore Theorem~\ref{th:MBfol:ltf} is a particular case of the following Theorem~\ref{th:Zfol:ltf}:

\begin{theorem}\label{th:Zfol:ltf}
Let $\Fol$ be a foliation of class $\Eclass$ on a manifold $\Mman$, e.g. a Morse-Bott foliation, $\Cman_1,\ldots,\Cman_k$ be any finite collection of its \myemph{compact} singular leaves, and $\Sigma = \mathop{\cup}\limits_{i=1}^{k}\Cman_i$, so $\Sigma$ is a submanifold of $\Mman$ whose connected components may have distinct dimensions.
Let $\Diff(\Sigma)$ be the group of diffeomorphisms of $\Sigma$.
Then for every closed subset $\Xman\subset\Mman \setminus \Sigma$ the homomorphisms
\begin{align}
\label{equ:restr_DFol}
&\rho: \Diff(\Fol,\Xman) \to \Diff(\Sigma), &
\rho(\dif)&= \dif|_{\Sigma}, \\
\label{equ:restr_DIdFol}
&\rho_0: \DiffId(\Fol,\Xman) \to \DiffId(\Sigma) \equiv \prod_{i=1}^{k} \DiffId(\Cman_i), &
\rho_0(\dif)&= (\dif|_{\Cman_1}, \ldots, \dif|_{\Cman_k}),
\end{align}
are locally trivial fibrations over their images and the map $\rho_0$ is surjective.
\end{theorem}
\begin{proof}
Notice that $\DiffId(\Sigma)$ is open in $\Diff(\Sigma)$.
Therefore it suffices to find an open neighborhood $\VV$ of $\id_{\Sigma}$ in $\DiffId(\Sigma)$ and a continuous map $\msect:\VV\to\DiffId(\Fol,\Xman)$ such that
\begin{align}\label{equ:need_to_prove}
\msect(\id_{\Sigma}) &= \id_{\Mman}, &
\rho_0\circ \msect &= \id_{\VV}.
\end{align}
Then by Lemma~\ref{lm:sect_hom} the maps~\eqref{equ:restr_DFol} and~\eqref{equ:restr_DIdFol} will be  locally trivial fibrations over their images.

For $i=1,\ldots,k$ let $\Nman_i$ be an open tubular neighborhood of $\Cman_i$ and $p_i:\Nman_i\to\Cman_i$ be a smooth vector bundle retraction.
One can assume that the closures $\{\overline{\Nman_i}\}_{i=1,\ldots,k}$ are compact and mutually disjoint.

Then for every open neighborhood $\aSet_i$ of $\Cman_i$ with $\overline{\aSet_i}\subset \Nman_i$ every $\dif\in \DfU{\aSet_i}{\Nman_i}{\Cman_i}$ is fixed outside $\aSet_i$, and so it extends by the identity to all of $\Mman$.
In other words, we have a natural inclusion:
\[
    j: \DfU{\aSet_i}{\Nman_i}{\Cman_i} \ \subset \ \Diff(\Mman, \Mman_i\setminus\Nman)
\]
which is continuous with respect to strong Whitney topologies since all $\dif\in\DfU{\aSet_i}{\Nman_i}{\Cman_i}$ are supported in the compact set $\overline{\Nman_i}$.

\begin{sublemma}\label{lm:sect_prop}
For each $i=1,\ldots,k$ there exist an open neighborhood $\VV_i$ of $\id_{\Cman_i}$ in $\DiffId(\Cman_i)$ and a continuous map $\msect_i:\VV_i\to \DiffId(\Fol,\Mman\setminus\Nman_i) \subset\DiffId(\Fol,\Xman)$ such that
\begin{enumerate}[label={\rm(\roman*)}]
\item\label{enum:sect_prop:id} $\msect_i(\id_{\Cman_i}) = \id_{\Mman}$;
\item\label{enum:sect_prop:Ci_inv} $\msect_i(\gdif)(\Cman_i) = \Cman_i$ and $\msect_i(\gdif)|_{\Cman_i} = \gdif$ for all $\gdif\in\VV_i$.
\end{enumerate}
\end{sublemma}
\begin{proof}
Using partition of unity technique one can glue local Ehresmann connections $\{\conn_x\}_{x\in\Cman_i}$ from Definition~\ref{def:Z:fol}\ref{enum:Z:Ehresmann_conn} and get a local Ehresmann connection $\conn_i$ over some open neighborhood $\aSet_i \subset \Nman_i$ of $\Cman_i$ in $\Nman_i$.
Moreover, it follows from Lemma~\ref{lm:Ehresmann_conn_convex} that $\HBundle{\conn_i} \subset T\Fol$ as well.

Then by Lemma~\ref{lm:sect_at_id} there exist an open path-connected neighborhood $\VV_i$ of $\id_{\Cman_i}$ and a homotopy
\[
    \rsect_i:\VV_i\times[0,1]\to\DfU{\aSet_i}{\Nman_i}{\Cman_i} \subset \Diff(\Mman, \Mman_i\setminus\Nman)
\]
having properties~\ref{enum:lm:sect_at_id:F0_pt}-\ref{enum:lm:sect_at_id:rho_xi__idU} of that lemma.

In particular, by properties~\ref{enum:lm:sect_at_id:F0_pt} and \ref{enum:lm:sect_at_id:loc_sect_adopted} for every $\gdif\in\VV_i$, $t\in[0,1]$, and $(x,v)\in\aSet_i$ the path $\gamma_{(x,v)}:t\mapsto\rsect_i(\gdif,t)(x,v)$ starts at $(x,v)$ and is horizontal, i.e.\! its tangent vectors are contained in the horizontal bundle $\HBundle{\conn_i}$ of the connection $\conn_i$ and therefore in the distribution $T\Fol$.
Hence every path $\gamma_{(x,v)}$ is contained in some leaf of the foliation $\Fol$.

Then $\rsect_i(\gdif,t)$ preserves leaves of $\Fol$, so $\rsect_i(\VV_i\times[0,1])\subset\Diff(\Fol,\Xman)$.
Moreover, as $\VV_i$ is path-connected, we get that $\rsect_i(\VV_i\times[0,1])\subset\DiffId(\Fol,\Xman)$.
One easily checks that the mapping $\msect_i = (\rsect_i)_1:\VV \equiv \VV\times 1 \to\DiffId(\Fol,\Xman)$ satisfies conditions~\ref{enum:sect_prop:id} and~\ref{enum:sect_prop:Ci_inv}.
\end{proof}

Define the following two maps
\begin{align*}
&\bar{\msect}: \VV_1\times\cdots\times \VV_k\to \DiffId(\Fol,\Xman), &
&\bar{\msect}(\gdif_1,\ldots,\gdif_k) = \msect_1(\gdif_1)\circ\cdots\circ\msect_k(\gdif_k), \\
&\psi:\VV_1\times\cdots\times \VV_k \to \DiffId(\Sigma), &
&\psi(\gdif_1,\ldots,\gdif_k)(x) =
\begin{cases}
\gdif_1(x), & x \in \Cman_1, \\
\cdots \\
\gdif_k(x), & x \in \Cman_k,
\end{cases}
\end{align*}
for $\gdif_i\in\VV_i$.
Then it is easy to verify that
\begin{itemize}
\item $\psi$ is an open embedding,
\item its image $\VV:=\psi(\VV_1\times\cdots\times\VV_k)$ is a neighborhood of $\id_{\Sigma}$ in $\DiffId(\Sigma)$,
\item $\bar{\msect}(\id_{\Cman_1},\ldots,\id_{\Cman_k})=\id_{\Mman}$,
\item $\rho\circ\bar{\msect} = \psi$.
\end{itemize}
Hence the map $\msect = \bar{\msect}\circ\psi^{-1}: \VV\to\DiffId(\Fol,\Xman)$ satisfies~\eqref{equ:need_to_prove}.
Theorem~\ref{th:Zfol:ltf} is completed.
\end{proof}

\section{Stabilizers of smooth maps $\Mman\to\Pman$}\label{sect:stabs_smooth_maps_MP}
Let $\Mman$ be a smooth compact manifold, $\Pman$ be either the real line or the circle $S^1$, $\func:\Mman\to\Pman$ be a smooth map, and $\fSing$ be the set of critical points of $\func$.
Then one can define a partition $\Fol_{\func}$ of $\Mman$ whose elements are path components of $\fSing$ and of sets $\func^{-1}(c)\setminus\fSing$, $c\in\Pman$.

\begin{lemma}
For every $\func\in \Ci{\Mman}{\Pman}$ and every subset $\Xman\subset\Mman$ we have that
\begin{align*}
&\Diff(\Fol_{\func}) \subset \Stabilizer{\func}, &
&\Diff(\Fol_{\func},\Xman) \subset \Stabilizer{\func,\Xman}, &
&\DiffId(\Fol_{\func},\Xman) \subset \StabilizerId{\func,\Xman}.
\end{align*}
\end{lemma}
\begin{proof}
By definition $\func$ takes constant values at elements of $\Fol_{\func}$, whence $\Diff(\Fol_{\func}) \subset \Stabilizer{\func}$.
Therefore
$\Diff(\Fol_{\func},\Xman) \, \equiv \,
\Diff(\Fol_{\func}) \cap \Diff(\Mman,\Xman) \, \subset \,
\Stabilizer{\func} \cap \Diff(\Mman,\Xman) \, \equiv \, \Stabilizer{\func,\Xman}$.
In particular, $\DiffId(\Fol_{\func},\Xman) \subset \StabilizerId{\func,\Xman}$.
\end{proof}

Notice that Lemmas~\ref{lm:fiberwise_morse_lemma} and~\ref{lm:bott_func_has_conn} imply that for every Morse-Bott map $\func:\Mman\to\Pman$ its foliation $\Fol_{\func}$ belongs to class $\Eclass$.
Therefore Theorem~\ref{th:MB_func:ltf} is a particular case of the following Theorem~\ref{th:func_ltf}.

\begin{theorem}\label{th:func_ltf}
Let $\Mman$ be a smooth manifold and $\func:\Mman\to\Pman$ be a smooth map whose foliation $\Fol_{\func}$ belongs to class $\Eclass$ with singularities $\{\Cman_i\}_{i\in\Lambda}$ and $\fSing = \mathop{\cup}\limits_{i\in\Lambda} \Cman_i$.
Let also and $\Cman_1,\ldots,\Cman_k$ be any finite collection of \myemph{compact} critical submanifolds of $\func$, and $\Xman\subset\Mman \setminus \Sigma$ be a closed (possibly empty) subset.
Then the maps
\begin{align*}
&\rho: \Stabilizer{\func,\Xman} \to \Diff(\Sigma), & \rho(\dif) &= \dif|_{\Sigma}, \\
&\rho_0: \StabilizerId{\func,\Xman} \to \DiffId(\Sigma) \equiv \prod_{i=1}^{k} \DiffId(\Cman_i), & \rho_0(\dif)&= (\dif|_{\Cman_1}, \ldots, \dif|_{\Cman_k}),
\end{align*}
are locally trivial fibrations over their images, and the map $\rho_0$ is surective.
\end{theorem}
\begin{proof}
By Theorem~\ref{th:Zfol:ltf} there exists an open neighborhood $\VV$ of $\id_{\Sigma}$ in $\DiffId(\Sigma)$ and a continuous map $\msect:\VV\to\DiffId(\Fol_{\func},\Xman) \subset \StabilizerId{\func,\Xman}$ such that $\msect(\id_{\Sigma}) = \id_{\Mman}$ and $\rho_0\circ \msect = \id_{\VV}$.
Then by Lemma~\ref{lm:sect_hom} the maps $\rho$ and $\rho_0$ are also locally trivial fibrations over their images.
\end{proof}

\section{Acknowledgement}
The authors are sincerely grateful to anonymous Referee for careful reading of the initial version of the manuscript and valuable remarks and suggestions allowed to clarify and improve the paper.


\def\cprime{$'$}

\end{document}